\documentclass[]{my-gOPT2e}

\begin{document}


\markboth{M. J. Lazo and D. F. M. Torres}{The Legendre Condition of the Fractional Calculus of Variations}

\title{\itshape The Legendre Condition of the Fractional Calculus of Variations}

\author{Matheus J. Lazo$^{\rm a}$
and Delfim F. M. Torres$^{\rm b}$$^{\ast}$\thanks{$^\ast$Corresponding author.
\vspace{6pt}}\\
$^{\rm a}${\em{Instituto de Matem\'{a}tica, Estat\'{\i}stica e F\'{\i}sica,
Funda\c{c}\~{a}o Universidade Federal do Rio Grande (FURG), Rio Grande, RS, Brasil.
Email: matheuslazo@furg.br}}\\
$^{\rm b}${\em{Center for Research and Development in Mathematics and Applications (CIDMA),
Department of Mathematics, University of Aveiro, 3810--193 Aveiro, Portugal.\\ Email: delfim@ua.pt}}\\
\vspace{6pt}\received{Submitted June 30, 2013; revised August 19, 2013; accepted December 16, 2013}}

\maketitle


\begin{abstract}
Fractional operators play an important role in modeling nonlocal phenomena
and problems involving coarse-grained and fractal spaces.
The fractional calculus of variations with functionals depending
on derivatives and/or integrals of noninteger order is a rather
recent subject that is currently in fast development due to its
applications in physics and other sciences.
In the last decade, several approaches to fractional variational calculus
were proposed by using different notions of fractional derivatives and integrals.
Although the literature of the fractional calculus of variations is already vast,
much remains to be done in obtaining necessary and sufficient conditions
for the optimization of fractional variational functionals, existence and regularity
of solutions. Regarding necessary optimality conditions, all works available
in the literature concern the derivation of
first-order fractional conditions of Euler--Lagrange type.
In this work we obtain a Legendre second-order necessary optimality condition
for weak extremizers of a variational functional that depends
on fractional derivatives.

\bigskip


\begin{classcode}
26A33; 49K05.
\end{classcode}
\bigskip
\end{abstract}


\section{Introduction}

The calculus with fractional derivatives and integrals of noninteger order
started more than three centuries ago with l'H\^opital and Leibniz,
when a derivative of order $\frac{1}{2}$ was suggested \cite{OldhamSpanier}.
The subject was then developed by several mathematicians like Euler, Fourier, Liouville,
Grunwald, Letnikov and Riemann, among many others, up to nowadays.
Although the fractional calculus is almost as old as the standard integer order calculus,
only in the last three decades it has gained more attention due to its applications
in various fields of science \cite{Diethelm,Hilfer,Kilbas,Magin,SATM,SKM}.
Fractional derivatives are generally nonlocal operators and have a great importance
in the study of nonlocal or time dependent processes, as well as to model phenomena
involving coarse-grained and fractal spaces. As an example, applications of fractional
calculus in coarse-grained and fractal spaces are found in the framework
of anomalous diffusion \cite{Klages,Metzler,Metzler2} and field theories
\cite{BI,Calcagni,Calcagni12,Lazo,Tarasov3,Vacaru}.

The fractional calculus of variations was introduced in 1996/1997
in the context of classical mechanics \cite{Riewe,Riewe2}.
In these two seminal papers, Riewe shows that a Lagrangian involving
fractional time derivatives leads to an equation of motion with
nonconservative forces such as friction. This is a remarkable result
since frictional and nonconservative forces are usually thought to be
beyond the usual macroscopic variational treatment \cite{Bauer} and,
consequently, beyond the most advanced methods of classical mechanics \cite{MyID:268}.
Riewe's approach generalizes the classical calculus of variations
for Lagrangians depending on fractional derivatives, and allows
to obtain conservation laws even in presence of nonconservative forces
\cite{MR2338631,MyID:257,MyID:268,MyID:262}.
Recently, several approaches have been developed to generalize
the least action principle and the Euler--Lagrange equations to include
fractional derivatives \cite{APT,AT,BA,Cresson,MR3019300,MyID:239,OMT}.
However, notwithstanding the literature of fractional calculus of variations is already vast
\cite{MyID:243,MalinowskaTorres,MyID:270}, all works concern, essentially, the obtention
of first-order conditions. The question of obtaining a Legendre second-order optimality
condition is an old open question in the fractional calculus of variations.
Attempts to address the problem always found many difficulties, mainly due to absence
of a good Leibniz rule \cite{tarasov:LR} and to the fact
that fractional derivatives are nonlocal operators.
The most successful result in this direction seems to be the one obtained by Bastos et al.
in \cite{MyID:179}, where a useful Legendre type condition was proved
for the discrete-time fractional calculus of variations on the time scale
$\mathbb{T} := h \mathbb{Z}$, $h > 0$. When the graininess function $h$ tends to zero,
the time scale $\mathbb{T}$ tends to the set $\mathbb{R}$ of real numbers
and the discrete operators introduced in \cite{MyID:179} to the fractional Riemann--Liouville
derivatives. Unfortunately, the obtained Legendre condition in \cite{MyID:179}
reduces, in the limit, to a truism without bringing any relevant information to the continuous case.
The authors of \cite{MyID:179} finish their paper asserting
``that a fractional Legendre condition for the continuous
fractional variational calculus is still an open question''. Here we solve the problem:
we obtain a generalization of the necessary Legendre condition for minimizers
of variational functionals dependent on Riemann--Liouville (or Caputo) fractional derivatives
(Theorem~\ref{thm:MR}).

The text is organized as follows.
In Section~\ref{sec:2} we review the basic definitions and properties
of Riemann--Liouville fractional derivatives, which are needed to formulate the fractional problem
of the calculus of variations. Our main result, the fractional Legendre second-order optimality condition,
is proved in Section~\ref{sec:3}. The usefulness of our condition
is then illustrated through an example in Section~\ref{sec:4}: the fractional Legendre condition
can be a practical tool in eliminating false candidates.
Finally, Section~\ref{sec:5} presents the main conclusions.


\section{The Riemann--Liouville Fractional Calculus}
\label{sec:2}

There are several definitions of fractional order derivatives.
Such definitions include the Riemann--Liouville, Caputo, Riesz,
Weyl, and Grunwald--Letnikov operators
\cite{Diethelm,Hilfer,Kilbas,Magin,OldhamSpanier,SATM,SKM}.
In this section we review some definitions and properties of the
Riemann--Liouville fractional calculus. Despite many different approaches
to fractional calculus, several known formulations are connected with
the analytical continuation of the Cauchy formula for $n$-fold integration.

\begin{theorem}
Let $f:[a,b]\rightarrow \mathbb{R}$ be Riemann integrable in $[a,b]$.
The $n$-fold integration of $f$, $n\in \mathbb{N}$, is given by
\begin{equation}
\label{a2}
\begin{split}
\int_{a}^x f(\tilde{x})(d\tilde{x})^{n}
&= \int_{a}^x\int_{a}^{x_{n}}\int_{a}^{x_{n-1}}
\cdots \int_{a}^{x_3}\int_{a}^{x_2} f(x_1)dx_1dx_2\cdots dx_{n-1}dx_{n}\\
&= \frac{1}{\Gamma(n)}\int_{a}^x \frac{f(u)}{(x-u)^{1-n}}du ,
\end{split}
\end{equation}
where $\Gamma$ is the Euler Gamma function.
\end{theorem}

The proof of Cauchy's formula \eqref{a2} can be found in several textbooks,
for example, it can be found in \cite{OldhamSpanier}. The analytical continuation
of \eqref{a2} gives a definition for integration of noninteger (or fractional) order.
This fractional order integration is the building block of the Riemann--Liouville calculus,
the most popular formulation of fractional calculus, as well as for several other approaches
\cite{Diethelm,Hilfer,Kilbas,Magin,OldhamSpanier,SATM,SKM}. The fractional integrations
obtained from \eqref{a2} are historically called the Riemann--Liouville fractional integrals.

\begin{definition}
Let $\alpha \in \mathbb{R}_+$.
The operators $_aJ^{\alpha}_x$ and $_xJ^{\alpha}_b$ defined on $L_1([a,b])$ by
\begin{equation}
\label{a3}
_aJ^{\alpha}_x f(x)
=\frac{1}{\Gamma(\alpha)}\int_{a}^x \frac{f(u)}{(x-u)^{1-\alpha}}du
\end{equation}
and
\begin{equation}
\label{a4}
_xJ^{\alpha}_b f(x)
=\frac{1}{\Gamma(\alpha)}\int_x^b \frac{f(u)}{(u-x)^{1-\alpha}}du,
\end{equation}
where $a,b\in \mathbb{R}$ with $a<b$, are called the left and the right
Riemann--Liouville fractional integrals of order $\alpha$, respectively.
\end{definition}

For an integer $\alpha$, the fractional Riemann--Liouville integral
\eqref{a3} coincide with the usual integer order
$n$-fold integration \eqref{a2}. Moreover, from definitions \eqref{a3}
and \eqref{a4}, it is easy to see that the Riemann--Liouville fractional
integrals converge for any integrable function $f$ if $\alpha>1$.
Furthermore, it is possible to prove the convergence of \eqref{a3}
and \eqref{a4} for $f\in L_1([a,b])$ even when $0<\alpha<1$ \cite{Diethelm}.

The integration operators $_aJ^{\alpha}_x$ and $_xJ^{\alpha}_b$
play a fundamental role in the definition of Riemann--Liouville fractional
derivatives. In order to define the fractional derivatives, we recall that
for positive integers $n>m$ it follows the identity
$D^m_x f(x)=D^{n}_x {_aJ}^{n-m}_x f(x)$, where $D^m_x$
is the ordinary derivative $\frac{d^m}{dx^m}$ of order $m$.

\begin{definition}
The left and the right Riemann--Liouville fractional derivatives
of order $\alpha \in \mathbb{R}_+$ are defined, respectively,
by ${_aD}^{\alpha}_x f(x) = D^{n}_x {_aJ}^{n-\alpha}_x f(x)$
and ${_xD}^{\alpha}_b f(x)=(-1)^n D^{n}_x {_xJ}^{n-\alpha}_b f(x)$,
where $n=[\alpha]+1$, that is,
\begin{equation}
\label{a5}
{_aD}^{\alpha}_x f(x)
=\frac{1}{\Gamma(n-\alpha)}\frac{d^n}{dx^n}
\int_{a}^x \frac{f(u)}{(x-u)^{1+\alpha-n}}du
\end{equation}
and
\begin{equation}
\label{a6}
{_xD}^{\alpha}_b f(x)
=\frac{(-1)^n}{\Gamma(n-\alpha)}\frac{d^n}{dx^n}
\int_{x}^b \frac{f(u)}{(u-x)^{1+\alpha-n}}du
\end{equation}
with $n$ the smallest integer greater than $\alpha$.
\end{definition}

An important consequence of definitions \eqref{a5} and \eqref{a6}
is that the Riemann--Liouville fractional derivatives
are nonlocal operators. The left (resp. right) differ-integration operator
\eqref{a5} (resp. \eqref{a6}) depend on the values of the function
at left (resp. right) of $x$, i.e., $a\leq u \leq x$ (resp. $x\leq u \leq b$).

\begin{remark}
When $\alpha$ is an integer, the Riemann--Liouville fractional derivatives
\eqref{a5} and \eqref{a6} reduce to ordinary derivatives of order $\alpha$.
\end{remark}

\begin{remark}
For $0<\alpha<1$, and because they are given by a first-order derivative of a fractional integral,
the Riemann--Liouville fractional derivatives \eqref{a5} and \eqref{a6} can be applied
to nondifferentiable functions. Actually, one can have Riemann--Liouville derivatives
of nowhere differentiable functions, like the Weierstrass function
\cite{KolwankarGangal,MR1313679}.
\end{remark}

\begin{remark}
It is important to notice that the Riemann--Liouville derivatives
\eqref{a5} and \eqref{a6} of a constant $\xi$ are not zero
for $\alpha \notin \mathbb{N}$. More precisely, one has
\begin{equation}
\label{a10}
{_a D}^{\alpha}_x \xi = \frac{\xi (x-a)^{-\alpha}}{\Gamma(1-\alpha)},
\quad {_x D}^{\alpha}_b \xi = \frac{\xi (b-x)^{-\alpha}}{\Gamma(1-\alpha)}.
\end{equation}
As $\alpha \to m \in \mathbb{N}$, the right hand sides of \eqref{a10} become equal to zero
due to the poles of the Gamma function. Furthermore, for the power function $(x-a)^{\beta}$, one has
\begin{equation}
\label{a11}
{_a D}^{\alpha}_x (x-a)^{\beta}
=\frac{\Gamma(\beta+1)}{\Gamma(\beta+1-\alpha)}(x-a)^{\beta-\alpha}.
\end{equation}
\end{remark}

Finally, as our goal is to obtain a fractional generalization of the second-order
Legendre condition, it is important to mention the fractional first-order
Euler--Lagrange equation \cite{MalinowskaTorres}.

\begin{theorem}
\label{thm:ml:03}
Let $J$ be a functional of the form
\begin{equation*}
J[y]=\int_a^b L\left(x,y(x),{_a D}^{\alpha}_x y(x)\right) dx
\end{equation*}
defined in the class of functions $y \in C([a,b])$
satisfying given boundary conditions  $y(a) = y_a$ and $y(b)=y_b$
and where $L$ is differentiable with respect to all of its arguments.
If $y^{\ast}$ is an extremizer of $J$, then $y^{\ast}$ satisfies
the following fractional Euler--Lagrange equation:
\begin{equation}
\label{t1}
\frac{\partial L\left(x,y(x),{_a D}^{\alpha}_x y(x)\right)}{\partial y}
+{_x D}^{\alpha}_b\frac{\partial
L\left(x,y(x),{_a D}^{\alpha}_x y(x)\right)}{\partial ({_a D}^{\alpha}_x y)}=0,
\end{equation}
$x \in [a,b]$.
\end{theorem}

\begin{remark}
Theorem~\ref{thm:ml:03} is often proved
under the hypothesis that admissible functions $y$
have continuous left fractional derivatives
on the closed interval $[a,b]$. Such assumption implies that
$y(a)=0$ (\textrm{cf.} \cite{MR1313679}).
Moreover, it is also common to impose some conditions
of continuity on the second addterm of \eqref{t1}.
The reader interested in these details is referred to \cite{MalinowskaTorres}
and references therein. For our current purposes it is enough to know that \eqref{t1}
holds under appropriate assumptions on the data of the problem.
\end{remark}

In analogy with the classical calculus of variations,
the solutions to \eqref{t1} are called extremals.


\section{The Legendre Necessary Optimality Condition}
\label{sec:3}

In order to generalize the classical Legendre necessary condition
\cite{vanBrunt} for functionals with Riemann--Liouville fractional
derivatives, we use the following lemma.

\begin{lemma}
\label{lemma}
Let $P$, $Q$ and $R$ be at least piecewise-continuous functions on $[a,b]$,
and $0<\alpha<1$. A necessary condition for
\begin{equation}
\label{l1}
\int_a^b\left[P(x) \, \left({_a D}^{\alpha}_x \eta(x)\right)^2
+Q(x) \, \eta(x) \, {_a D}^{\alpha}_x \eta(x)
+ R(x) \, \left(\eta(x)\right)^2\right]dx \geq 0
\end{equation}
for all $\eta \in C([a,b])$ satisfying $\eta(a)=\eta(b)=0$
is that $P(x)\geq 0$ for all $x \in [a,b]$.
\end{lemma}

\begin{proof}
Suppose, by contradiction, that the lemma is false.
Then, since $P(x)$ is at least piecewise-continuous,
there must exist a subinterval $[c,c+d] \subset [a,b]$,
$d>0$, throughout which $P(x)<-K_1$, where $K_1>0$.
Likewise, because $Q(x)$ and $R(x)$ are at least piecewise-continuous,
there must exist constants $K_2$ and $K_3$ such that $Q(x)< K_2$ and $R(x)< K_3$
for all  $x \in [c,c+d]$. Define the function
\begin{equation}
\label{eq:fuct:f}
f(x) = \frac{1}{d}(x-c)-\frac{5-\alpha}{2d^2}(x-c)^2
+\frac{3-\alpha}{2d^3}(x-c)^3.
\end{equation}
From \eqref{a11} it follows that
\begin{equation*}
{_c D}^{\alpha}_x f(x)=\frac{1}{\Gamma(2-\alpha)}\left[
\frac{1}{d}(x-c)^{1-\alpha}-\frac{5-\alpha}{\left(2-\alpha\right)d^2}
(x-c)^{2-\alpha}+\frac{3}{(2-\alpha)d^3}(x-c)^{3-\alpha}\right]
\end{equation*}
and direct calculations show that
$f(c)=f(c+d)={_c D}^{\alpha}_c f(x)={_c D}^{\alpha}_{c+d} f(x)=0$.
We also note that
\begin{equation*}
|f(x)|<1 \quad \mbox{ and }
\quad \left|{_c D}^{\alpha}_x f(x)\right|<\frac{1}{d^{\alpha}}
\end{equation*}
for all $x\in [c,c+d]$ and $d>0$. Let us now define
\begin{equation}
\label{l3}
\tilde{\eta}(x)=
\begin{cases}
0 & \text{ if } \ a\leq x <c, \\
f(x) & \text{ if } \ c \leq x \leq c+d,\\
0 & \text{ if } \ c+d < x \leq b.
\end{cases}
\end{equation}
Then, $\tilde{\eta}\in C([a,b])$ and $\tilde{\eta}(a)=\tilde{\eta}(b)=0$.
Furthermore, since ${_c D}^{\alpha}_{c+d} f(x)=0$, we have
\begin{equation}
\label{l4}
{_a D}^{\alpha}_x \tilde{\eta}(x)=
\begin{cases}
0 & \text{ if } \ a\leq x <c, \\
{_c D}^{\alpha}_x f(x) & \text{ if } \ c \leq x \leq c+d,\\
0 & \text{ if } \ c+d < x \leq b.
\end{cases}
\end{equation}
By inserting \eqref{l3} and \eqref{l4} into the integral \eqref{l1}, we obtain
\begin{equation*}
\begin{split}
\int_a^b \bigl[P(x) &\left({_a D}^{\alpha}_x \tilde{\eta}(x)\right)^2
+Q(x) \tilde{\eta}(x) {_a D}^{\alpha}_x \tilde{\eta}(x)
+ R(x) \left(\tilde{\eta}(x)\right)^2\bigr]dx\\
&= \int_c^{c+d}\left[P(x) \left({_c D}^{\alpha}_x f(x)\right)^2
+Q(x) f(x) {_c D}^{\alpha}_x f(x)
+ R(x) \left(f(x) \right)^2\right]dx\\
&< d^{1-2\alpha}\left[-K_1+K_2d^{\alpha} +K_3d^{2\alpha}\right],
\end{split}
\end{equation*}
which is negative for sufficiently small $d$ and thus a contradiction.
\end{proof}

Follows the main result of the paper: the fractional Legendre condition.

\begin{theorem}
\label{thm:MR}
For the extremal $y^{\ast}\in C([a,b])$ to yield
a weak relative minimum (resp. maximum) to functional
\begin{equation}
\label{t2}
J[y]=\int_a^b L\left(x,y(x),{_aD}^{\alpha}_x y(x)\right)dx
\end{equation}
defined in $\left\{y \, | \, y \in C([a,b]), y(a)=y_a, y(b)=y_b\right\}$,
where $L\in C^2\left([a,b]\times \mathbb{R}^2\right)$ and
$\alpha \in (0,1)$, it is necessary that
\begin{equation}
\label{t3}
\frac{\partial^2 L\left(x,y^{\ast}(x),{_aD}^{\alpha}_x y^{\ast}(x)\right)}{
\partial \left({_aD}^{\alpha}_x y\right)^2} \geq 0
\quad (\text{resp. } \leq 0)
\end{equation}
for all $x\in [a,b]$.
\end{theorem}

\begin{proof}
Let $y^{\ast}$ give a minimum to
$\displaystyle J[y]=\int_a^b L\left(x,y(x),{_aD}^{\alpha}_x y(x)\right)dx$.
We define a family of functions
\begin{equation}
\label{b10}
y(x)=y^{\ast}(x)+\epsilon \eta(x),
\end{equation}
where $\epsilon$ is a constant and $\eta \in C([a,b])$
is an arbitrary continuous function
satisfying the boundary conditions
$\eta(a)=\eta(b)=0$ (weak variations).
From \eqref{b10} and the boundary conditions $\eta(a)=\eta(b)=0$
and $y^{\ast}(a)=y_a$, $y^{\ast}(b)=y_b$,
it follows that function $y$ is admissible:
$y \in C([a,b])$ with $y(a)=y_a$ and $y(b)=y_b$.
Because $y^{\ast}$ is a minimizer of functional $J$, we
should have $J[y]-J[y^{\ast}]\geq 0$ for all $y$ defined by \eqref{b10}.
Let the Lagrangian $L$ be $C^2([a,b]\times \mathbb{R}^2)$. In this case,
we have $J[y]-J[y^{\ast}]\geq 0$ only if (see, for example, Theorem~1.8.2 of \cite{Sagan})
\begin{equation}
\label{b11}
\delta^2 J[y]=\frac{d^2}{d\epsilon^2}
J[y^{\ast}+\epsilon \eta]_{\epsilon=0}\geq 0,
\end{equation}
where $\delta^2 J[y]$ is the second G\^ateaux variation of $J[y]$.
For \eqref{t2} the necessary optimality condition \eqref{b11} asserts that
\begin{equation}
\label{b12}
\delta^2 J[y]
=\int_a^b\left[P(x) \left({_a D}^{\alpha}_x \eta(x)\right)^2
+ Q(x) \eta(x) {_a D}^{\alpha}_x \eta(x)
+ R(x) \left(\eta(x)\right)^2\right]dx \geq 0,
\end{equation}
where the functions
\begin{eqnarray*}
P(x)&=&\frac{\partial^2 L\left(x,y^{\ast}(x),{_aD}^{\alpha}_x y^{\ast}(x)\right)}{
\partial \left({_{a}D}^{\alpha}_x y\right)^2},\\
Q(x)&=& 2 \, \frac{\partial^2 L\left(x,y^{\ast}(x),{_aD}^{\alpha}_x y^{\ast}(x)\right)}{
\partial y\partial \left({_aD}^{\alpha}_x y\right)},\\
R(x)&=&\frac{\partial^2 L\left(x,y^{\ast}(x),{_aD}^{\alpha}_x y^{\ast}(x)\right)}{\partial y^2},
\end{eqnarray*}
are continuous since $L\in C^2\left([a,b]\times \mathbb{R}^2\right)$.
Because $\eta \in C([a,b])$, the fractional Legendre necessary optimality condition \eqref{t3}
follows from \eqref{b12} and Lemma~\ref{lemma}.
\end{proof}


\section{An Illustrative Example}
\label{sec:4}

Let $\alpha \in (0,1)$. Consider the functional
\begin{equation}
\label{b13}
J[y]=\int_{-1}^{1} x\sqrt{1+\left({_{-1}D}^{\alpha}_x y(x)\right)^2}dx
\end{equation}
subject to given boundary conditions
\begin{equation}
\label{b14}
y(-1)=y_{-1}, \quad y(1)=y_{1}.
\end{equation}
We can write the associated Euler--Lagrange equation \eqref{t1}
and try to find the corresponding extremals
to problem \eqref{b13}--\eqref{b14}.
However, using our Legendre condition \eqref{t3},
we can immediately conclude that the Euler--Lagrange extremals
are not local extremizers. Indeed, the Lagrangian
is given by $L=x\sqrt{1+\left({_{-1}D}^{\alpha}_xy\right)^2}$ and
\begin{equation*}
\frac{\partial^2L}{\partial \left({_{-1}D}^{\alpha}_x y\right)^2}
=\frac{x}{\sqrt{\left(1+\left({_{-1}D}^{\alpha}_xy\right)^2\right)^3}},
\end{equation*}
so that the sign of $\frac{\partial^2L}{\partial \left({_{-1}D}^{\alpha}_x y\right)^2}$
changes in the interval $[-1,1]$: it is nonpositive for  $x\in [-1,0]$ and nonnegative
for $x\in [0,1]$.


\section{Conclusions}
\label{sec:5}

The fractional calculus of variations, with functionals dependent
on fractional derivatives and integrals, has experienced a huge development
in the last few years due to its many applications
in several different areas \cite{MalinowskaTorres}.
Despite the quite vast literature in this recent field, most of the works
concern the obtention of first-order fractional optimality conditions.
In this context, much remains to be done. Main difficulties reside
in the absence of a good Leibniz rule and on the nonlocal properties
of the fractional operators. Here we provided the first generalization
of Legendre's second-order necessary optimality condition
to the continuous fractional calculus of variations.
The usefulness of the new condition is illustrated with an example.
Although the obtained result is formulated for functionals dependent
on Riemann--Liouville fractional derivatives, the techniques
here introduced are still valid for other types
of fractional operators, in particular for Caputo derivatives:
the function \eqref{eq:fuct:f} used here for
Riemann--Liouville derivatives is also valid
for Caputo derivatives because for this polynomial $f$
the derivatives coincide, the same being true
for function $\tilde{\eta}$ given by \eqref{l3}.


\section*{Acknowledgements}

M. J. Lazo was partially supported
by CNPq and CAPES (Brazilian research funding agencies);
D. F. M. Torres by FCT
(Portuguese Foundation for Science and Technology)
through CIDMA and project PEst-C/MAT/UI4106/2011
with COMPETE number FCOMP-01-0124-FEDER-022690.

The authors would like to thank the anonymous referees
for their careful reading of the submitted
manuscript and for suggesting useful changes.



\end{document}